\documentclass[11pt]{article}
\usepackage{}

\usepackage{epsfig}
\usepackage{latexsym}
\usepackage{caption}
\usepackage{amsfonts}
\usepackage{amssymb}
\usepackage{mathrsfs}
\usepackage{amsmath}
\usepackage{enumerate}
\usepackage{graphics}
\usepackage{MnSymbol}
\usepackage{float}
\usepackage{pict2e}
\usepackage{subfigure}

\usepackage{color}

\makeatletter

\newcommand{\Rmnum}[1]{\expandafter\@slowromancap\romannumeral #1@}

\makeatother

\begin{document}

\newtheorem{theorem}{Theorem}
\newtheorem{observation}[theorem]{Observation}
\newtheorem{corollary}[theorem]{Corollary}
\newtheorem{algorithm}[theorem]{Algorithm}
\newtheorem{definition}[theorem]{Definition}
\newtheorem{guess}[theorem]{Conjecture}
\newtheorem{claim}[theorem]{Claim}
\newtheorem{problem}[theorem]{Problem}
\newtheorem{question}[theorem]{Question}
\newtheorem{lemma}[theorem]{Lemma}
\newtheorem{proposition}[theorem]{Proposition}
\newtheorem{fact}[theorem]{Fact}

\makeatletter
  \newcommand\figcaption{\def\@captype{figure}\caption}
  \newcommand\tabcaption{\def\@captype{table}\caption}
\makeatother

\newtheorem{acknowledgement}[theorem]{Acknowledgement}

\newtheorem{axiom}[theorem]{Axiom}
\newtheorem{case}[theorem]{Case}
\newtheorem{conclusion}[theorem]{Conclusion}
\newtheorem{condition}[theorem]{Condition}
\newtheorem{conjecture}[theorem]{Conjecture}
\newtheorem{criterion}[theorem]{Criterion}
\newtheorem{example}[theorem]{Example}
\newtheorem{exercise}[theorem]{Exercise}
\newtheorem{notation}{Notation}
\newtheorem{solution}[theorem]{Solution}
\newtheorem{summary}[theorem]{Summary}

\newenvironment{proof}{\noindent {\bf
Proof.}}{\rule{3mm}{3mm}\par\medskip}
\newcommand{\remark}{\medskip\par\noindent {\bf Remark.~~}}
\newcommand{\pp}{{\it p.}}
\newcommand{\de}{\em}
\newcommand{\mad}{\rm mad}
\newcommand{\qf}{Q({\cal F},s)}
\newcommand{\qff}{Q({\cal F}',s)}
\newcommand{\qfff}{Q({\cal F}'',s)}
\newcommand{\f}{{\cal F}}
\newcommand{\ff}{{\cal F}'}
\newcommand{\fff}{{\cal F}''}
\newcommand{\fs}{{\cal F},s}
\newcommand{\s}{\mathcal{S}}
\newcommand{\G}{\Gamma}
\newcommand{\g}{(G_3, L_{f_3})}
\newcommand{\wrt}{with respect to }
\newcommand {\nk}{ Nim$_{\rm{k}} $  }

\newcommand{\q}{\uppercase\expandafter{\romannumeral1}}
\newcommand{\qq}{\uppercase\expandafter{\romannumeral2}}
\newcommand{\qqq}{\uppercase\expandafter{\romannumeral3}}
\newcommand{\qqqq}{\uppercase\expandafter{\romannumeral4}}
\newcommand{\qqqqq}{\uppercase\expandafter{\romannumeral5}}
\newcommand{\qqqqqq}{\uppercase\expandafter{\romannumeral6}}

\newcommand{\qed}{\hfill\rule{0.5em}{0.809em}}

\newcommand{\var}{\vartriangle}

\title{{\large \bf The  Alon-Tarsi number of planar graphs - a simple proof    }}

\author{ Yangyan Gu\thanks{Department of Mathematics, Zhejiang Normal University,  China.  E-mail: yangyan@zjnu.edu.cn.} \and Xuding Zhu\thanks{Department of Mathematics, Zhejiang Normal University,  China.  E-mail: xdzhu@zjnu.edu.cn.  NSFC 11971438,U20A2068, ZJNSFC LD19A010001.}}

\maketitle

\begin{abstract}
	
	This paper gives a simple proof  of the result  that every planar graph $G$ has Alon-Tarsi number at most 5, and  has a matching $M$ such that $G-M$ has Alon-Tarsi number at most 4. 
	 

\end{abstract}


Assume $D$ is an orientation of a graph $G$. 
For a subset $H$ of arcs of $D$, $D[H]$ denotes the sub-digraph induced by $H$. 
We say $D[H]$ is an {\em Eulerian sub-digraph} if   $d_{D[H]}^+(v)=d_{D[H]}^-(v)$ for each vertex $v$. 
Let  $$\mathcal{E}(D) = \{H: D[H] \text{ is an Eulerian sub-digraph}\},$$ 
 $$\mathcal{E}_e(D) = \{H \in \mathcal{E}(D): |H| \text{ is even}\},  \ \mathcal{E}_o(D) = \{H \in \mathcal{E}(D): |H| \text{ is odd}\},  {\rm diff}(D)= |\mathcal{E}_e(D)|-|\mathcal{E}_o(D)|.$$ 
We  say $D$ is an {\em Alon-Tarsi orientation}   if ${\rm diff}(D) \ne 0$. 

Let $\mathbb{N}=\{0,1,2,\ldots\}$ and $\mathbb{N}^G = \{f: V \to \mathbb{N}\}$. For $f,g \in \mathbb{N}^G$, we write $g \le f$ if $g(v) \le f(v)$ for all $v \in V(G)$. 
For  $f  \in \mathbb{N}^G$, an {\em $f$-Alon-Tarsi orientation} (an $f$-AT orientation, for short) of   $G$ is an Alon-Tarsi orientation $D$ of $G$ with $d_D^+(v) \le f(v)-1$ for each vertex $v$. We say $G$ is {\em $f$-AT} if  $G$ has an $f$-AT orientation.   The {\em Alon-Tarsi number} $AT(G)$ of $G$ is the minimum $k$ such that $G$ is $k$-AT (i.e., $f$-AT for the constant function $f(v)=k$ for all $v$).
Let $ch(G)$ be the choice number of $G$, and $\chi_P(G)$ be the paint number of $G$. It is well-known \cite{ZB} that for any graph $G$, 
$$ch(G) \le \chi_P(G) \le AT(G).$$ 
It is known that for any  planar graph $G$,  
   $ch(G) \le 5$ \cite{TH},     $\chi_P(G) \le 5$ \cite{Scha},        $AT(G) \le 5$   \cite{Zhu}, and  $G$ has a matching $M$ such that $AT(G-M) \le 4$ \cite{GZ}. This note gives a simple proof of the last two results. Note that $AT(G) \le 5$   implies that $ch(G) \le \chi_P(G) \le 5$.

For $f \in  \mathbb{N}^G$ and $X \subseteq V(G)$, let $f_{[X,-1]} \in  \mathbb{N}^G$ be defined as
\[
f_{[X,-1]}(v) = \begin{cases} f(v)-1, &\text{ if $v \in X$},  \cr
f(v), &\text{ otherwise.}
\end{cases}
\]   We write $ f_{[v,-1]}$ for $f_{[\{v\},-1]}$. 

The following lemma is folklore (cf. Proof of Lemma 3.1.7 in \cite{ZB}). 
\begin{lemma}
	\label{lem-subgraph}
	If $G$ is $f$-AT and $e=uv \in E(G)$, then $G-e$ is $f_{[u,-1]}$-AT or $f_{[v,-1]}$-AT.
\end{lemma}

 \begin{observation}
 	 \label{obs-1}
If $G$ is $f$-AT, and $D$ is an $f$-AT orientation of $G$, then $f(x) \ge d_D^+(x)+1 \ge 1$ for each vertex $x$. If $f(x)=1$, then $d_D^+(x) =0$, and hence any neighbour $y$ of $x$ has $d_D^+(y) \ge 1$ and $f(y) \ge 2$. 
	 Therefore the following hold:
	 \begin{itemize}
	 	\item[(A)]  If $G$ is $f$-AT, $e=uv \in E(G)$, and $f(v)=1$, then $G-e$ is $f_{[u,-1]}$-AT.
	 	\item[(B)] If $G$ is $f$-AT, $e=uv, e'=vw \in E(G)$, and $f(v)=2$, $f(w)=1$, then $G-e$ is $f_{[u,-1]}$-AT.
	 	\end{itemize}	 
	 \end{observation}

\begin{corollary}
	\label{cor-1}
	Assume $d_G(w)=2$, $N_G(w)=\{u,v\}$, $f: V(G) \to \mathbb{N}$  and $f(w)=2$. If $G$ is $f$-AT,  then $G-w$ is $f_{[u,-1]}$-AT or $f_{[v,-1]}$-AT.
\end{corollary}
\begin{proof}
	By  Lemma \ref{lem-subgraph},  $G-wu$ is $f_{[u,-1]}$-AT or $f_{[w,-1]}$-AT. If $G-wu$ is $f_{[u,-1]}$-AT, then   $G-w$ is $f_{[u,-1]}$-AT. If  $G-wu$ is $f_{[w,-1]}$-AT, then since $f_{[w,-1]}(w)=1$, 
	$G-w=G-\{wu, wv\}$ is $f_{[v,-1]}$-AT.
\end{proof}

The following lemma is also easy and well-known (cf. \cite{ZB}).

\begin{lemma}
	\label{lem1}
	Assume $D$ is an orientation of $G=(V,E)$, $X \subseteq V$ and all arcs between $X$ and $V-X$ are oriented from $X$ to $V-X$. Then ${\rm diff} (D) = {\rm diff}(D[X]) \times {\rm diff}(D[V-X])$, where $D[X], D[V-X]$ are sub-digraphs induced by $X$ and $V-X$.
\end{lemma}

 \begin{definition}
 	\label{def-nice}
 	Assume $G$ is a 2-connected plane graph and $e_1=v_1v_2$ is a boundary edge.   Assume $M$ is a matching of $G$ containing $e_1$. Let $f_{G, v_1v_2}, f_{G, v_1v_2,M} \in \mathbb{N}^G$ be defined as 
 	\[
 	f_{G, v_1v_2}(x) = \begin{cases} 1, &\text{ if $x \in \{v_1,v_2\}$}, \cr
 	3, &\text{if $x \in B(G)-\{v_1,v_2\}$} \cr
 	5, &\text{ if $x \in V(G)-B(G)$}.\cr
 	\end{cases}
 	\]
 	\[
 	f_{G, v_1v_2, M}(x) = \begin{cases} 1, &\text{ if $x \in \{v_1,v_2\}$}, \cr
 	3-d_M(x), &\text{if $x \in B(G)-\{v_1,v_2\}$} \cr
 	4, &\text{  if $x \in V(G)-B(G)$}.\cr
 	\end{cases}
 	\]
 \end{definition}
 
 The result that stated in the abstract follows from the following more technical theorem.
 
 \begin{theorem}
 	\label{thm-main}
 	Assume  $G$ is a 2-connected plane graph and $e_1=v_1v_2$ is a boundary edge. Then $G-e_1$ is $f_{G, v_1v_2}$-AT, and $G$ has a matching $M$ containing $e_1$ such that $G-M$ is $f_{G, v_1v_2,M}$-AT.  
 \end{theorem}
 \begin{proof}
 The proof is by induction on $|V(G)|$. If $G$ is a triangle $(v_1,v_2,v_3)$, then let $D$ be the orientation of $G-e_1$ with $v_3$ be a source. Then $D$ is an $f_{G,v_1v_2}$-AT orientation of $G-e_1$, and  also  an $f_{G,v_1v_2,M}$-AT orientation $G-M$ with $M=\{e_1\}$. 
 	
Assume $|V(G)| \ge 4$ and Theorem \ref{thm-main} is true for smaller plane graphs. 

\bigskip
\noindent
{\bf Case 1} $B(G)$ has a chord $e=xy$. 

Let $G_1,G_2$ be the subgraphs of $G$ separated by $e$ (i.e., $V(G_1) \cap V(G_2) = \{x,y\}$ and $V(G_1) \cup V(G_2) = V(G)$), and assume $v_1v_2 \in E(G_1)$. By induction hypothesis, $G_1-e_1$ has an $f_{G_1,v_1v_2}$-AT orientation $D_1$ and $G_2-xy$ has $f_{G_2,xy}$-AT orientation $D_2$.  Let $D=D_1 \cup D_2$. It follows from Lemma \ref{lem1} that ${\rm diff} (D) = {\rm diff}(D_1) \times  {\rm diff}(D_2) \ne 0$.  
Hence $D$ is an $f_{G,v_1v_2}$-AT orientation of $G-e_1$.

Also by induction hypothesis, $G_1$ has a matching $M_1$ containing $e_1$ so that $G_1-M_1$ has an $f_{G_1,v_1v_2, M_1}$-AT orientation $D'_1$, and $G_2$ has a matching $M_2$ containing $e=xy$ so that $G_2-M_2$ has an $f_{G_2,xy, M_2}$-AT orientation $D'_2$. 
 	Let $M=M_1 \cup (M_2 -\{e\})$ and let $D'=D'_1 \cup D'_2$.  Then $M$ is a matching of $G$ containing $e_1$.  By Lemma \ref{lem1},  ${\rm diff}(D') = {\rm diff} (D'_1) \times {\rm diff}(D'_2) \ne 0$, and hence $D'$ is an $f_{G,v_1v_2, M}$-AT orientation of $G-M$.

 \bigskip
 \noindent
 {\bf Case 2} $B(G)$ has no chord. 	
 
 Assume $B(G)=(v_1,v_2, \ldots, v_n)$ and let $G'=G-v_n$. Assume $N_G(v_n) =\{v_1, u_1,u_2, \ldots, u_k, v_{n-1}\}$.  Let $U = \{u_1,u_2, \ldots,u_k\}$. By induction hypothesis,   $G'-e_1$ has an $f_{G',v_1v_2}$-AT orientation $D_1$.
 Let $G''$ be obtained from $G$ by adding paths
 $P_1,P_2,\ldots, P_k$, where $P_i=(u_i,w_i,v_n)$, and $w_i$ are new vertices. 
 Let $b_1=(v_n,v_{n-1}), b_2=(v_n, v_1)$    and for $i=1,2,\ldots,k$, $a_i=(u_i,v_n), a'_i=(u_i, w_i), a''_i=(w_i, v_n)$. Let $$Z= 
 \{a_i, a'_i, a''_i: i=1,2,\ldots, k\} \cup \{b_1,b_2\}.$$ 
 	Let $D_2$  be the orientation of $G''-e_1$ obtained from $D_1$ by adding arcs in $Z$. 
 For an arc $a$ of $D_2$, let 
 	$$\mathcal{E}(D_2, a)= \{H \in \mathcal{E}(D_2): a \in H\}.$$ 
 	
 It is easy to see that  
  $$\mathcal{E}(D_2) = \mathcal{E}(D_1) \cup (\cup_{i=1}^k  (\mathcal{E}(D_2, a_i) \cup 
 	\mathcal{E}(D_2, a'_i)) ),$$
 	where   the unions are disjoint  unions.  
The map $\phi: \mathcal{E}(D_2, a_i) \to \mathcal{E}(D_2, a'_i)$ defined as $\phi(H) = (H - a_i) \cup \{a'_i,a''_i\}$ is a one-to-one correspondence, and $H$ and $\phi(H)$ have opposite parities. So 
$$|\mathcal{E}_e(D_2) \cap (\cup_{i=1}^k  (\mathcal{E}(D_2, a_i) \cup 
\mathcal{E}(D_2, a'_i)) )| =  |\mathcal{E}_o(D_2) \cap (\cup_{i=1}^k ( \mathcal{E}(D_2, a_i) \cup 
\mathcal{E}(D_2, a'_i)) )|$$  and hence   ${\rm diff}(D_2) = {\rm diff}(D_1) \ne 0.$
Let $g(v) =d_{D_2}^+(v)+1$ for $v \in V(G'')$. Then $G''-e_1$ is $g$-AT, and by Lemma \ref{lem-subgraph}, $G-e_1$ is $g$-AT. It is easy to see that $g \le f_{G, v_1v_2}$. Hence $G-e_1$ is $f_{G, v_1v_2}$-AT.

Similarly, by induction hypothesis, 
$G'$ has a matching $M_1$ containing $e_1$ so that $G'-M_1$ has an $f_{G',v_1v_2, M_1}$-AT orientation $D'_1$.
Let   $D'_2$  be the orientation of $G''-M_1$ obtained from $D'_1$   by adding arcs in $Z$.  
For the same reason as above, we have $ {\rm diff}(D'_2) = {\rm diff}(D'_1) \ne 0$.
Thus  $D'_2$ is an Alon-Tarsi orientation of $G''-M_1$. 

Let  $g(v) = d_{D'_2}^+(v)+1$. Then $G''-M_1$ is $g$-AT. It follows from the construction of $D'_2$ that 
 $g(v) \le f_{G, v_1v_2, M_1}(v)$ for $ v \in V(G)-U.$ 
   For $i=1,2,\ldots, k$, if $d_{D'_1}^+(u_i) \le 1$, then  
we also have $g(u_i) = d_{D'_1}^+(u_i) +2+1 \le 4 =    f_{G, v_1v_2, M_1}(u_i)$. In particular, if $u_i$ is covered by the matching $M_1$, then  $g(u_i) \le  f_{G, v_1v_2, M_1}(u_i)$.  However, if $d_{D'_1}^+(u_i) =2$, then $g(u_i)= 5=f_{G, v_1v_2, M_1}(u_i)+1$.

Let $G_k=G''$ and for $i=0, 1, \ldots, k-1$, let $G_i = G_{i+1}-w_{i+1}$. Apply  Corollary \ref{cor-1} to $G_{k-1}, G_{k-2}, \ldots, G_0$ one by one, we conclude that for $i=k-1,  \ldots, 0$, 
$G_i-M_1$ is $g_i$-AT, where $g_k=g$ and  
for $i=k-1,  \ldots, 0$, either (i) $g_i = g_{i+1, [u_{i+1},-1]}$,  or (ii) $g_i = g_{i+1, [v_{n},-1]}$. 
 
Note that 
$g(v_1)=1$ and $g(v_n)=3$. By (B) of Observation \ref{obs-1}, $g_i(v_n) \ge 2$ for all $i$. Hence   (ii) occurs at most once. If (ii) never occurs, then $G_0-M_1=G-M_1$ is  $g_{[U,-1]}$-AT. As $g_{[U,-1]} \le f_{G, v_1v_2, M_1}$, we are done. 

Assume (ii) occurs exactly once at some $i$. Then $G_0-M_1=G-M_1$ is  $g_{[(U-\{u_{i}\}) \cup \{v_n\},-1]}$-AT. 
If $u_i$ is covered by $M_1$, then we have $g_{[(U-\{u_{i}\}) \cup \{v_n\},-1]} \le f_{G, v_1v_2, M_1}$, we are done. Otherwise, let $M=M_1 \cup \{u_iv_n\}$.
Note that if $D$ is a $g_{[(U-\{u_{i}\}) \cup \{v_n\},-1]}$-orientation of $G-M_1$, then $v_n$ has only one out-neighbour $v_1$, which is a sink. So any arc in $D$ incident to $v_n$ is not contained in any Eulerian sub-digraph of $D$. Hence $D-(u_i,v_n)$ is   an AT-orientation of $G-M$.  So  $G-M$ is $g_{[ U  \cup \{v_n\},-1]}$-AT. As $g_{[ U  \cup \{v_n\},-1]} \le f_{G, v_1v_2,M}$, we have $G-M$ is $ f_{G, v_1v_2,M}$-AT. 
	\end{proof}
	
\bibliographystyle{unsrt}

\end{document}